\documentclass[11pt,a4paper]{article}
\usepackage{amssymb}
\usepackage{eurosym}
\usepackage{amsfonts}
\usepackage{amsmath}
\usepackage{amsthm}
\usepackage{graphicx}
\usepackage{float}
\usepackage{hyperref}

\hypersetup{
	colorlinks=true,
	linkcolor=blue,
	anchorcolor=blue,
	citecolor=blue
}
\newtheorem{theorem}{Theorem}[section]

\newtheorem{corollary}[theorem]{Corollary}
\newtheorem{lemma}[theorem]{Lemma}

\theoremstyle{definition}

\newtheorem{remark}[theorem]{Remark}

\renewenvironment{proof}[1][Proof]{\noindent\textbf{#1.} }{\ \rule{0.5em}{0.5em}}
\newenvironment{acknowledgement}{\smallskip{\sc Acknowledgement.}\rm}{\smallskip}
\renewcommand{\theequation}{\thesection.\arabic{equation}}
\allowdisplaybreaks
\input{tcilatex}

\def\func#1{\mathop{\mathrm{#1}}\nolimits}

\def\dint{\displaystyle\int}

\def\Xint#1{\mathchoice
{\XXint\displaystyle\textstyle{#1}}%
{\XXint\textstyle\scriptstyle{#1}}%
{\XXint\scriptstyle\scriptscriptstyle{#1}}%
{\XXint\scriptscriptstyle\scriptscriptstyle{#1}}%
\!\int}
\def\XXint#1#2#3{{\setbox0=\hbox{$#1{#2#3}{\int}$ }
\vcenter{\hbox{$#2#3$ }}\kern-.6\wd0}}

\def\oint{\Xint-}

\def\enddoc{\end{document}}

\def\FRAME#1#2#3#4#5#6#7#8
{
 \begin{figure}[H]
 \begin{center}
 \includegraphics[height=#3]{#7}
 \caption{#5}
 \label{#6}
 \end{center}
 \end{figure}
}

\begin{document}
	\title{Gradient estimates for Leibenson's equation on Riemannian manifolds}
	\author{Philipp S\"urig}
	\date{June 2025}
	\maketitle
	
	\begin{abstract}
		We consider on Riemannian manifolds solutions of the Leibenson equation%
		\begin{equation*}
			\partial _{t}u=\Delta _{p}u^{q}.
		\end{equation*}%
		This equation is also known as doubly nonlinear evolution equation. We prove gradient estimates for positive solutions $u$ under the condition that the Ricci curvature on $M$ is bounded from below by a non-positive constant. We distinguish between the case $q(p-1)>1$ (slow diffusion case) and the case $q(p-1)<1$ (fast diffusion case).
	\end{abstract}
	
	\let\thefootnote\relax\footnotetext{
		The author was funded by the Deutsche Forschungsgemeinschaft (DFG,
		German Research Foundation) - Project-ID 317210226 - SFB 1283.}
	
	\tableofcontents
	
	\section{Introduction}
	
We are concerned here with a non-linear evolution
equation 
\begin{equation}
\partial _{t}u=\Delta _{p}u^{q}  \label{olddtvint}
\end{equation}%
where $p>1$, $q>0$, $u=u(x,t)$ is an unknown non-negative function and $%
\Delta _{p}$ is the $p$-Laplacian 
\begin{equation*}
\Delta _{p}w=\func{div}\left( |\nabla w|^{p-2}\nabla w\right) .
\end{equation*}

The equation (\ref{olddtvint}) is frequently referred to as a \emph{doubly non-linear parabolic equation}.
In the case when $p=2$, equation (\ref{olddtvint}) becomes the \textit{porous medium equation} $\partial _{t}u=\Delta u^{q}$, and the classical heat equation $\partial _{t}u=\Delta u$ if also $q=1$.

We consider (\ref{olddtvint}) on Riemannian manifolds. Let $M$ be a geodesically complete Riemannian manifold of dimension $n$. Let $B=B\left( x_{0},R\right)$ be a geodesic ball of radius $R>0$. Let the Ricci curvature $Ric_{B}$ in $B$ be bounded from below by $-K$ for some $K\geq 0$. Then Li and Yau \cite{li1986parabolic} proved that any positive solution $u$ of the heat equation $\partial _{t}u=\Delta u$ in $B\times (0, \infty)$ satisfies, for any $\alpha>1$ and all $t>0$, \begin{equation}\label{Liyau}\sup_{B}\left( \frac{|\nabla u|^{2}}{u^{2}}-\alpha\frac{\partial_{t}u}{u}\right)\leq \frac{\alpha^{2}n}{t}+ \frac{C\alpha^{2}}{R^{2}} \left(\frac{\alpha^{2}}{\alpha-1}+\sqrt{K}R\right)+\frac{\alpha^{2}nK}{2(\alpha-1)},
\end{equation} where $C$ depends on $n$. In particular, if the Ricci curvature on $M$ is bounded from below by $-K$ for some $K\geq 0$, letting $R\to \infty$ in (\ref{Liyau}), yields the global estimate \begin{equation}\label{liyauglobal}\sup_{M}\left( \frac{|\nabla u|^{2}}{u^{2}}-\alpha\frac{\partial_{t}u}{u}\right)\leq\frac{\alpha^{2}n}{t}+ \frac{\alpha^{2}nK}{2(\alpha-1)}.\end{equation}
Later the estimate (\ref{liyauglobal}) was slightly improved by Davies \cite{davies1989heat}.

In the present paper we obtain similar results for the Leibenson equation (\ref{olddtvint}). In the first part of the paper we assume that \begin{equation}\label{deltacondint}\delta:=q(p-1)-1>0.\end{equation}
For any positive smooth solution $u$ of (\ref{olddtvint}) in a domain of $M$, set \begin{equation}\label{pressureint}v=\frac{q(p-1)}{\delta}u^{\frac{\delta}{p-1}}.\end{equation} 
The first main result of the present paper (cf. \textbf{Theorem \ref{mainthm}}) is as follows.

	\begin{theorem}\label{mainthmint}
	Let $B=B\left( x_{0},R\right) $ be a geodesic ball of radius  $R>0$ and let $Ric_{B}\geq-K$ for some $K\geq 0$. Let $u$ be a positive smooth solution of (\ref{olddtvint}) in $B\times (0, \infty)$ and $v$ be defined by (\ref{pressureint}).
	Assume that in $B\times (0, \infty)$, \begin{equation}\label{Lambdaminmaxint}\Lambda_{min}\leq  |\nabla v|^{p-2}v\leq \Lambda_{max},\end{equation} for some positive constants $\Lambda_{min}, \Lambda_{max}.$
		Then we have for all $t>0$ and any $\alpha>1$,
	\begin{align}\label{mainlineint}
		\sup_{B}\left( \frac{|\nabla v|^{p}}{v}-\alpha\frac{\partial_{t}v}{v}\right)\leq& \left(\frac{Ct}{\Lambda_{min}}\right)^{\frac{C^{\prime}}{1+\sqrt{K}R}}C_{0} \left[\frac{1}{t}+\frac{\Lambda_{max}\left(1+\sqrt{K}R\right)}{R^{2}}\right]^{1+\frac{C^{\prime}}{1+\sqrt{K}R }}\\\nonumber&+\frac{\alpha^{2}n K\delta^{2}\Lambda_{max}}{(p-1)(\alpha-1)},
	\end{align} where $C_{0}=\frac{\left[2(p-1)a_{p}+a_{p}^{2}\delta+16(p-1)^{2}\delta\right]\alpha^{2}n\delta}{\left[p+\alpha n\delta\right](p-1)^{2}a_{p}}$, $C=C(p, q, n, \alpha)$, $C^{\prime}=C^{\prime}(p, q, n, \alpha)$ and $a_{p}=\min(p-1, 1)$.
	\end{theorem}

By sending $R\to \infty$ in (\ref{mainlineint}) we get the following

	\begin{corollary}\label{cormainint}
	Let $u$ be a positive smooth solution of (\ref{olddtvint}) in $M\times (0, \infty)$. Assume that $Ric_{M}\geq-K$ for some $K\geq 0$ and that (\ref{Lambdaminmaxint}) holds in $M\times (0, \infty)$. Then for all $t>0$ and any $\alpha>1$, \begin{equation}\label{corglobint}\sup_{M}\left( \frac{|\nabla v|^{p}}{v}-\alpha\frac{\partial_{t}v}{v}\right)\leq\frac{C_{0}}{t}+\frac{\alpha^{2}n K\delta^{2}\Lambda_{max}}{(p-1)(\alpha-1)},\end{equation} where $C_{0}$ is as in Theorem \ref{mainthmint}.
	\end{corollary}

It remains an open question for further work if in Corollary \ref{cormainint} condition (\ref{Lambdaminmaxint}) in $M\times (0, \infty)$ can be relaxed to the condition that only $|\nabla v|^{p-2}v\leq \Lambda_{max}$ holds in $M\times (0, \infty)$ for some positive constant $\Lambda_{max}$.

Let us now discuss previously known results.

Consider first the case when $u$ is a positive solution of the general Leibenson equation (\ref{olddtvint}). Under condition (\ref{deltacondint}) it is proved in \cite{chen2015gradient} by Chen and Xiong that if $M$ is a \textit{closed} manifold with $Ric_{M}\geq-K$ for some $K\geq 0$, then, for all $t>0$ and any $\alpha>1$, \begin{equation}\label{closedres}\sup_{M}\left( \frac{|\nabla v|^{p}}{v}-\alpha\frac{\partial_{t}v}{v}\right)\leq \frac{\alpha^{2}n\delta}{p(p-1)t}+\frac{\alpha^{2} nK\delta^{2}\Lambda_{max}}{4(\alpha-1)} .\end{equation}
Comparing this result with our global estimate (\ref{corglobint}), we see that the right hand side of both estimates go to zero for $\delta \to 0$ and blow up in the limit  $\delta\to \infty$. One of the advantages of the result of the present paper is that it holds for \textit{non-compact} manifolds. 

In the case when $p=2$, that is, (\ref{olddtvint}) becomes the porous medium equation, the following was known. Assuming that $q>1$ (which is equivalent to $\delta>0$), Lu, Ni, Vázquez, and Villani proved in \cite{lu2009local} that on an arbitrary geodesically complete manifold with $Ric_{M}\geq-K$, for some $K\geq 0$, it holds that \begin{equation}\label{Vazglobal}\sup_{M}\left( \frac{|\nabla v|^{2}}{v}-\alpha\frac{\partial_{t}v}{v}\right)\leq\frac{\alpha^{2}n\delta}{(n\delta+2)t}+\frac{\alpha^{2}nK\delta^{2}\Lambda_{max}}{(n\delta+2)(\alpha-1)}.\end{equation}
Like in our estimate (\ref{corglobint}), in the limit $\delta\to 0$, the right hand side of (\ref{Vazglobal}) goes to zero. In the limit $\delta \to \infty$, the first summond of the right hand side of (\ref{Vazglobal}) stays bounded, while the corresponding term in our estimates blows up for $\delta \to \infty$. However, the method in \cite{lu2009local} is seriously linked to the case $p=2$, while our result also covers the case $p\ne2$.

Other global estimates for solutions of the porous medium equation under the same conditions on the manifold were later obtained by G. Huang, Z. Huang and H. Li in \cite{huang2013gradient} and S. Huang and Shen in \cite{huang2025gradient}. The aforementioned papers \cite{huang2013gradient, huang2025gradient, lu2009local} contain also the gradient estimates in balls similar to our estimate (\ref{mainlineint}).

	In the second part of the present paper we are concerned with the Leibenson equation (\ref{olddtvint}) under the condition that $\delta<0$, that is, when \begin{equation}\label{deltacondintD}D:=-\delta=1-q(p-1)>0.\end{equation} 
	Let us also assume that \begin{equation}\label{FDEcond}p-nD>0.\end{equation}
For any positive smooth solution $u$ of (\ref{olddtvint}) in a domain of $M$, define $v$ by \begin{equation}\label{newpressure}v=\frac{q(p-1)}{D}u^{-\frac{D}{p-1}}.\end{equation}.

Our second main result (cf. \textbf{Theorem \ref{mainthmD}}) is as follows.

\begin{theorem}\label{mainthmintD}
	Let $B=B\left( x_{0},R\right) $ be a geodesic ball of radius  $R>0$ and assume that $Ric_{B}\geq-K$ for some $K\geq 0$. Let $u$ be a positive smooth solution of (\ref{olddtvint}) in $B\times (0, \infty)$ and $v$ be defined by (\ref{newpressure}).
	Assume that in $B\times (0, \infty)$, \begin{equation}\label{LambdaminmaxintD}\Lambda_{min}\leq  |\nabla v|^{p-2}v\leq \Lambda_{max},\end{equation} for some positive constants $\Lambda_{min}, \Lambda_{max}$.
	Then we have for all $t>0$ and any $0<\alpha<1$,
	\begin{align}\label{mainlineintD}
		\sup_{B}\left(\frac{|\nabla v|^{p}}{v}+\alpha\frac{\partial_{t}v}{v}\right)\leq& \left(\frac{Ct}{\Lambda_{min}}\right)^{\frac{C^{\prime}}{1+\sqrt{K}R}}C_{0} \left[\frac{1}{t}+\frac{\Lambda_{max}\left(1+\sqrt{K}R\right)}{R^{2}}\right]^{1+\frac{C^{\prime}}{1+\sqrt{K}R }}\\\nonumber&+C^{\prime\prime}K\Lambda_{max},
	\end{align} where $C$, $C_{0}$, $C^{\prime}$ and $C^{\prime\prime}$ are all positive constants that depend on $p, q, n, \alpha$.
\end{theorem}

By sending $R\to \infty$ in (\ref{mainlineintD}) we get the following

\begin{corollary}\label{cormainintD}
	Let $u$ be a positive smooth solution of (\ref{olddtvint}) in $M\times (0, \infty)$. Assume that $Ric_{M}\geq-K$ for some $K\geq 0$ and that (\ref{LambdaminmaxintD}) holds in $M\times (0, \infty)$. Then for all $t>0$ and any $0<\alpha<1$, \begin{equation}\label{corglobintD}\sup_{M}\left(\frac{|\nabla v|^{p}}{v}+ \alpha\frac{\partial_{t}v}{v}\right)\leq\frac{C_{0}}{t}+C^{\prime\prime}K\Lambda_{max},\end{equation} where $C_{0}$ and $C^{\prime\prime}$ are as in Theorem \ref{mainthmintD}.
\end{corollary}

For the porous medium equation, conditions (\ref{deltacondintD}) and (\ref{FDEcond}) amount to $\frac{n-2}{n}< q<1$.
Under these conditions, similar estimates for the porous medium equation on Riemannian manifolds were proved in \cite{ huang2025gradient, lu2009local}. In the case of the parabolic $p$-Laplace equation, conditions (\ref{deltacondintD}) and (\ref{FDEcond}) become $\frac{2n}{n+1}<p<2$. In this case, global gradient estimates were proved in \cite{kotschwar2009local} on Riemannian manifolds with \textit{non-negative} Ricci curvature. As far as we know, Theorem \ref{mainthmintD} and Corollary \ref{cormainintD} are the first results of this type for a general Leibenson equation on manifolds with $Ric_{M}\geq-K$ for some $K\geq 0$.

Let us describe the structure of the present paper.

In Section \ref{secslow} we prove our results in the case when (\ref{deltacondint}) holds (slow diffusion case). In Subsection \ref{preliminaries} we introduce an operator $$\mathcal{F}=\partial_{t}-\frac{\delta}{p-1}v\mathcal{L},$$ where  $\mathcal{L}$ is a certain elliptic operator. By using a non-linear Bochner formula for $\mathcal{F}$ and the Ricci curvature assumption (cf. Lemma \ref{bochnerlem}), we get in Lemma \ref{lemmavarphi} an upper bound for $\mathcal{F}(f)$, where $$f=\frac{|\nabla v|^{p}}{v}-\alpha\frac{\partial_{t}v}{v}-\varphi,$$ and $\varphi$ is a specific auxiliary function in $\mathbb{R}_{+}$ (see Subsection \ref{secauxfunc} in the Appendix for function $\varphi$).
From that we obtain in Lemma \ref{lemcacc} a Caccioppoli type inequality for $f$.

In Subsection \ref{SecMV} we prove in Lemma \ref{LemMoser} the main technical tool-  a $L^{\lambda}$ mean value inequality for $f$ and large enough $\lambda$. This inequality says the following. Let $B=B\left( x_{0},R\right) $ be a geodesic ball and $T>0$. Let $u$ be a positive smooth solution of (\ref{olddtvint}) in $B\times \lbrack 0,T\rbrack$.
Let us set 
\begin{equation*}
	Q=B\times \left[ 0,T\right] .
\end{equation*}%
Assume that \begin{equation*}\Lambda_{min}\leq \Lambda(v, \nabla v)\leq \Lambda_{max}\quad \textnormal{in}~Q\end{equation*}
and that $Ric_{B}\geq -K$, for some $K\geq 0$.
Then, for the cylinder 
\begin{equation*}
	Q^{\prime }=\frac{1}{2}B\times [ \frac{1}{2}T,T] ,
\end{equation*}%
we have for all large enough $\lambda>1$,
\begin{equation*}
	\left\Vert f\right\Vert _{L^{\infty }(Q^{\prime })}\leq \left(\frac{CS_{B}}{\Lambda_{min}}\left[\frac{1}{T} +\frac{\Lambda_{max}}{R^{2}}\right]^{1+\nu}\right) ^{\frac{1}{\lambda\nu }}\left\Vert f\right\Vert _{L^{\lambda }(Q)},  \label{v31int}\end{equation*}%
where $C=C\left( p,q, n\right) $. Here $S_{B}$ and $\nu $ are
positive constants that depend on the intrinsic geometry of the ball $B$,
namely, on the Sobolev inequality in $B$ (see Section \ref{secMoser}). The proof of this mean value inequality follows a version of the Moser iteration argument \cite{Moser} by utilizing the Caccioppoli type inequality and a \textit{Moser inequality} (Lemma \ref{MoserLem}).

We have borrowed the idea of applying the Moser iteration argument to get our local gradient estimate (\ref{mainlineint}) from \cite{huang2025gradient}, where the authors dealt with solutions of the porous medium equation. This idea appeared first in \cite{wang2010local}, where it was used to obtain similar estimates for solutions of the elliptic $p$-Laplace equation.

Next, by using again the Caccioppoli type inequality, we prove in Lemma \ref{normbound} an upper bound for the $L^{\lambda}$-norm of $f$ in $Q$ in terms of the Riemannian measure of the cylinder $Q$. Combining this bound with the mean value inequality and an upper bound of $\varphi$, we conclude Theorem \ref{mainthm} (Theorem \ref{mainthmint}).

In Section \ref{fastdiff} we prove our results in the case of fast diffusion (that is, when (\ref{deltacondintD}) and (\ref{FDEcond}) hold). The method of the proof of our main theorem in this section- Theorem \ref{mainthmD}, follows a very similar approach as the method from Section \ref{secslow} in the slow diffusion case.

For other qualitative and quantitative properties of solutions of the Leibenson equation (\ref{olddtvint}) under condition (\ref{deltacondint}) on Riemannian manifolds, see \cite{grigor2023finite, Grigoryan2024, meglioli2025global} and under the conditions (\ref{deltacondintD}) and (\ref{FDEcond}) see \cite{grigor2025upper}.

We denote by $C$ and $C^{\prime}$ a positive constant whose value might change at each occurance.

\begin{acknowledgement}
	The author would like to thank Alexander Grigor'yan for many helpful discussions.
\end{acknowledgement}

	\section{Slow diffusion case}\label{secslow}

	\subsection{Functional inequalities}\label{preliminaries}
	
	We consider in what follows the following evolution equation on a Riemannian
	manifold $M$:%
	\begin{equation}
		\partial _{t}u=\Delta _{p}u^{q}.  \label{olddtv}
	\end{equation}

	We assume throughout that 
	\begin{equation*}
		p>1\ \ \text{and}\ \ \ q>0.
	\end{equation*}% 
	Set $$\delta=q(p-1)-1.$$
	Let $u$ be a positive smooth solution of (\ref{olddtv}). In this section we always assume that $\delta >0$ and define \begin{equation}\label{pressure}v=\frac{q(p-1)}{\delta}u^{\frac{\delta}{p-1}}.\end{equation} 
	%Then $v$ satisfies the equation \begin{equation}\label{asinlinear}\partial_{t} v=\frac{\delta}{p-1}v\Delta_{p}v+|\nabla v|^{p}.\end{equation}
	For any smooth function $\psi$ let us set $$\mathcal{L}(\psi)=\func{div}\left( |\nabla v|^{p-2}A(\nabla \psi)\right),$$ where $$A(\nabla \psi)=\nabla \psi+(p-2)|\nabla v|^{-2}\langle \nabla v, \nabla \psi\rangle \nabla v.$$
	Let us consider the operator $$\mathcal{F}=\partial_{t}-\frac{\delta}{p-1}v\mathcal{L}.$$
	
	Let $\Omega$ be an open subset of $M$ and $I$ be an interval in $[0, \infty)$.
	
	\begin{lemma}\label{bochnerlem}\cite{wang2014gradient}
	Let $u$ be a positive smooth solution of (\ref{olddtv}) in $\Omega\times I$ and $v$ be defined by (\ref{pressure}). Assume that $\textnormal{Ric}_{\Omega}\geq -K$, for some $K\geq0$. Then, for any $\alpha>1$, \begin{equation}\label{DefF1}F_{\alpha}=\frac{|\nabla v|^{p}}{v}-\alpha\frac{\partial_{t}v}{v}=:y-\alpha z\end{equation} satisfies \begin{equation}\label{diffequation}\mathcal{F}(F_{\alpha})\leq\beta|\nabla v|^{p-2}\langle \nabla v, \nabla F_{\alpha}\rangle-c_{n}(y-z)^{2}-(p-1)(\alpha-1)z^{2}+c_{\delta}K|\nabla v|^{2(p-1)},\end{equation} where $\beta=2\delta+p$, $c_{n}=\frac{(p-1)(n\delta+p)}{n\delta}$ and $c_{\delta}=\frac{p\delta}{p-1}$.
	\end{lemma}

\begin{lemma}\label{lemmavarphi}
Let $c_{0}=\frac{(p-1)(\alpha-1)[p(\alpha-1)+\alpha n\delta]}{n\delta\alpha^{2}}$,  $c_{1}=\frac{(p-1)[p+\alpha n\delta]}{n\delta\alpha^{2}}$ and $c_{2}=\frac{2(p-1)p(\alpha-1)}{n\delta\alpha^{2}}$. Let $\varphi$ be the non-negative continuous function from Lemma \ref{lemmauxfunc} (see Appendix) with $A=c_{1}$, $a=\frac{c_{2}}{2\sqrt{c_{0}}}$ and $b=\frac{c_{\delta}K\Lambda_{max}}{2\sqrt{c_{0}}}$.
Under the assumptions of Lemma \ref{bochnerlem}, set \begin{equation}\label{deff}f(x, t)=F_{\alpha}(x, t)-\varphi(t).\end{equation}
Assume also that \begin{equation}\label{defLamb}\Lambda(v, \nabla v):=|\nabla v|^{p-2}v\leq \Lambda_{max}.\end{equation} Then we have \begin{equation}\label{diffequationsecond}\mathcal{F}(f)\leq\beta|\nabla v|^{p-2}\langle \nabla v, \nabla f\rangle-c_{1}f^{2}-c_{2}yf.\end{equation}
\end{lemma}

When $K=0$, we can take $\varphi\equiv0$.

\begin{proof}
From (\ref{diffequation}) we have that $$\mathcal{F}(f)\leq\beta|\nabla v|^{p-2}\langle \nabla v, \nabla f\rangle-c_{n}(y-z)^{2}-(p-1)(\alpha-1)z^{2}+c_{\delta}K|\nabla v|^{2(p-1)}-\varphi(t)^{\prime} .$$ Using that $z=\frac{1}{\alpha}(y-f-\varphi)$, we get that \begin{align*}
c_{n}(y-z)^{2}+(p-1)(\alpha-1)z^{2}=c_{0}y^{2}+c_{2}yf+c_{2}y\varphi+c_{1}f^{2}+2c_{1}f\varphi+c_{1}\varphi^{2}.
\end{align*}
Hence, $$\mathcal{F}(f)\leq\beta|\nabla v|^{p-2}\langle \nabla v, \nabla f\rangle-c_{0}y^{2}-c_{2}yf-c_{2}y\varphi-c_{1}f^{2}-2c_{1}f\varphi-c_{1}\varphi^{2}+c_{\delta}K|\nabla v|^{2(p-1)}-\varphi(t)^{\prime}  .$$

We want that \begin{equation}\label{condforphi}c_{0}y^{2}+c_{2}y\varphi+2c_{1}f\varphi+c_{1}\varphi^{2}-c_{\delta}K|\nabla v|^{2(p-1)}+\varphi(t)^{\prime}\geq 0\end{equation}
Using that $2c_{1}f\varphi\geq 0$ and $|\nabla v|^{2(p-1)}\leq \Lambda_{max}y$ we see that (\ref{condforphi}) is satisfied if $$c_{0}y^{2}+\left(c_{2}\varphi-c_{\delta} K \Lambda_{max}\right)y+c_{1}\varphi^{2}+\varphi^{\prime}\geq 0.$$
This holds true if and only if one of the following two conditions are satisfied: $$(1):\quad c_{1}\varphi^{2}+\varphi^{\prime}\geq 0\quad \textnormal{and}\quad c_{2}\varphi-c_{\delta} K \Lambda_{max}\geq 0$$ or $$(2):\quad \left(c_{2}\varphi-c_{\delta} K \Lambda_{max}\right)^{2}-4c_{0} \left(c_{1}\varphi^{2}+\varphi^{\prime}\right)\leq 0.$$ Noticing that by Lemma \ref{lemmauxfunc} condition (1) is satisfied by $\varphi$ for $0<t<\frac{a(1-\varepsilon)}{2bc_{1}}$ and condition (2) is satisfied by $\varphi$ for $t\geq\frac{a(1-\varepsilon)}{2bc_{1}} $, we obtain (\ref{diffequationsecond}). 
\end{proof}
	
\begin{remark}\normalfont
By Remark \ref{remdel} we have that \begin{equation}\label{uppervarphi}\varphi(t)\leq 2\frac{b}{a}=\frac{c_{\delta}K\Lambda_{max}}{c_{2}}=\frac{\alpha^{2}n K\delta^{2}\Lambda_{max}}{(p-1)(\alpha-1)}.\end{equation}	
\end{remark}
		
	Let $\mu $ denote the Riemannian measure on $M$. For simplicity of notation,
	we frequently omit in integrations the notation of measure. All integration
	in $M$ is done with respect to $d\mu $, and in $M\times \mathbb{R}$ -- with
	respect to $d\mu dt$, unless otherwise specified.
		
	\begin{lemma}\label{lemcacc}
	Let $f$ be defined by (\ref{deff}) in a cylinder $\Omega\times I$. Assume that $\textnormal{Ric}_{\Omega}\geq -K$, for some $K\geq0$. Let $\eta \left( x,t\right) $ be a locally Lipschitz non-negative bounded
	function in $\Omega\times I$ such that $\eta \left( \cdot ,t\right) $ has
	compact support in $\Omega $ for all $t\in I$. Fix some large enough real $\lambda=\lambda(p, q, n, \alpha)>1$.
	Choose $t_{1},t_{2}\in I$ such that $t_{1}<t_{2}$ and set $Q=\Omega \times \left[ t_{1},t_{2}%
	\right] $. Then%
	\begin{align}\nonumber
		\left[ \int_{\Omega }f^{\lambda }\eta ^{2}\right] _{t_{1}}^{t_{2}}+&c_{3}
		\int_{Q}\left\vert \nabla \left( f^{\lambda/2 }\eta \right) \right\vert
		^{2}\Lambda(v, \nabla v)+\lambda c_{1}\int_{Q}f^{\lambda+1}\eta^{2}\\&\leq C_{1}\int_{Q}\left[\eta \partial _{t}\eta+\left\vert \nabla \eta \right\vert ^{2}\Lambda(v, \nabla v)\right]f^{\lambda } ,\label{Cacciotype}
	\end{align}%
	where $\Lambda$ is defined by (\ref{defLamb}), $c_{3}=\frac{a_{p}\delta}{p-1}$, $c_{1}$ is as in Lemma \ref{lemmavarphi}, $C_{1}=\frac{2(p-1)a_{p}+a_{p}^{2}\delta+16(p-1)^{2}\delta}{(p-1)a_{p}}$ and $a_{p}=\min(p-1, 1)$.
\end{lemma}

\begin{proof}
Multiplying both sides of (\ref{diffequationsecond}) with test function $\psi=f^{\lambda -1}\eta^{2}$ and integrating over $Q$ yields \begin{equation}\label{firststepinproof}\int_{Q}\left(\partial_{t}f-\frac{\delta}{p-1}v\mathcal{L}(f)\right)\psi\leq \int_{Q}\left(\beta|\nabla v|^{p-2}\langle \nabla v, \nabla f\rangle-c_{1}f^{2}-c_{2}yf\right)\psi.\end{equation} Using partial integration we obtain \begin{align*}-\int_{Q}v\mathcal{L}(f)\psi=\int_{Q}|\nabla v|^{p-2}\langle \nabla f, \nabla (v\psi)\rangle+(p-2)|\nabla v|^{p-4}\langle \nabla v, \nabla f\rangle \langle \nabla f, \nabla (v\psi)\rangle.\end{align*} Since $$\nabla \psi=(\lambda-1)f^{\lambda-2}\nabla f\eta^{2}+2f^{\lambda-1}\eta \nabla \eta,$$ we have \begin{align*}-\int_{Q}v\mathcal{L}(f)\psi=&\int_{Q}(\lambda-1)|\nabla v|^{p-2}f^{\lambda-2}v|\nabla f|^{2}\eta^{2}+2|\nabla v|^{p-2}vf^{\lambda-1}\eta\langle \nabla f, \nabla \eta\rangle\\&+ |\nabla v|^{p-2}f^{\lambda-1}\eta^{2}\langle \nabla f, \nabla v\rangle+(p-2)(\lambda-1)|\nabla v|^{p-4}\langle \nabla f, \nabla v\rangle^{2}f^{\lambda-2}\eta^{2}v\\&+(p-2)2|\nabla v|^{p-4}\langle \nabla f, \nabla v\rangle vf^{\lambda-1}\eta \langle \nabla v, \nabla \eta\rangle +(p-2) |\nabla v|^{p-2}\langle \nabla f, \nabla v\rangle f^{\lambda-1}\eta^{2}.\end{align*} Since \begin{align*}(\lambda-1)|\nabla v|^{p-2}f^{\lambda-2}v|\nabla f|^{2}\eta^{2}+&(p-2)(\lambda-1)|\nabla v|^{p-4}\langle \nabla f, \nabla v\rangle^{2}f^{\lambda-2}\eta^{2}v\\&\geq a_{p}(\lambda-1)|\nabla v|^{p-2}f^{\lambda-2}v|\nabla f|^{2}\eta^{2},\end{align*} where $a_{p}=1 $ if $p>2$ and $a_{p}=p-1 $ if $p<2$, we deduce \begin{align*} -\int_{Q}v\mathcal{L}(f)\psi\geq \int_{Q}a_{p}(\lambda-1)|\nabla v|^{p-2}f^{\lambda-2}v|\nabla f|^{2}\eta^{2}&-2(p-1)f^{\lambda-1}|\nabla v|^{p-2}v \eta |\nabla f||\nabla \eta|\\&+(p-1)|\nabla v|^{p-2}f^{\lambda-1}\eta^{2}\langle \nabla f, \nabla v\rangle.\end{align*}
Also, we have $$\int_{Q}\partial_{t}f\psi=\frac{1}{\lambda}\int_{Q}\partial_{t} f^{\lambda}\eta^{2}=\frac{1}{\lambda}\left[ \int_{\Omega }f^{\lambda }\eta ^{2}\right] _{t_{1}}^{t_{2}}-\frac{2}{\lambda}\int_{Q} f^{\lambda}\eta\partial_{t}\eta.$$
For the right hand side of (\ref{firststepinproof}), we have \begin{align*} \int_{Q}&\left(\beta|\nabla v|^{p-2}\langle \nabla v, \nabla f\rangle-c_{1}f^{2}-c_{2}yf\right)\psi\\&\leq \int_{Q}\beta|\nabla v|^{p-1}|\nabla f|f^{\lambda-1}\eta^{2}-c_{1}f^{\lambda+1}\eta^{2}-c_{2}yf^{\lambda}\eta^{2}.\end{align*}
Hence, combining these terms, we get
\begin{align*}
\frac{1}{\lambda}&\left[ \int_{\Omega }f^{\lambda }\eta ^{2}\right] _{t_{1}}^{t_{2}}+\frac{a_{p}\delta(\lambda-1)}{p-1}\int_{Q}|\nabla v|^{p-2}f^{\lambda-2}v|\nabla f|^{2}\eta^{2}+c_{1} \int_{Q}f^{\lambda+1}\eta^{2}+c_{2}\int_{Q}yf^{\lambda}\eta^{2}\\&\leq 2\delta \int_{Q}f^{\lambda-1}|\nabla v|^{p-2}v \eta |\nabla f||\nabla \eta|+(\beta-\delta)\int_{Q}f^{\lambda-1}|\nabla v|^{p-1} \eta^{2} |\nabla f|+\frac{2}{\lambda}\int_{Q}f^{\lambda}\eta\partial_{t}\eta.
\end{align*}

Further, we have by Young's inequality \begin{align*}2\delta& f^{\lambda-1}|\nabla v|^{p-2}v \eta |\nabla f||\nabla \eta|\\&\leq  \frac{a_{p}\delta(\lambda-1)}{4(p-1)}\Lambda(v, \nabla v)f^{\lambda-2}|\nabla f|^{2}\eta^{2}+\frac{C_{1}}{\lambda-1}\Lambda(v, \nabla v)f^{\lambda}|\nabla \eta|^{2},\end{align*} and \begin{align*}(\delta+p)&f^{\lambda-1}|\nabla v|^{p-1} \eta^{2} |\nabla f|\\&\leq \frac{a_{p}\delta(\lambda-1)}{4(p-1)}\Lambda(v, \nabla v)f^{\lambda-2}|\nabla f|^{2}\eta^{2}+\frac{C_{2}}{\lambda-1}yf^{\lambda}\eta^{2},\end{align*} where $C_{1}=\frac{16(p-1)\delta}{a_{p}}$ and $C_{2}=\frac{4(p-1)(\delta+p)^{2}}{a_{p}\delta}$.
Therefore, \begin{align*}\left[ \int_{\Omega }f^{\lambda }\eta ^{2}\right] _{t_{1}}^{t_{2}}+&\frac{a_{p}\delta\lambda(\lambda-1)}{2(p-1)}\int_{Q}\Lambda(v, \nabla v)f^{\lambda-2}|\nabla f|^{2}\eta^{2}+c_{1}\lambda  \int_{Q}f^{\lambda+1}\eta^{2}+c_{2}\lambda\int_{Q}yf^{\lambda}\eta^{2}\\&\leq \frac{\lambda C_{1}}{\lambda-1}\int_{Q}\Lambda(v, \nabla v)f^{\lambda}|\nabla \eta|^{2}+\frac{\lambda C_{2}}{\lambda-1}\int_{Q}yf^{\lambda}\eta^{2}+2\int_{Q}f^{\lambda}\eta\partial_{t}\eta.
\end{align*}
Choosing $\lambda$ large enough such that $$\frac{ C_{2}}{\lambda-1}<\frac{ c_{2}}{2}$$ we get \begin{align*}\left[ \int_{\Omega }f^{\lambda }\eta ^{2}\right] _{t_{1}}^{t_{2}}+&\frac{a_{p}\delta\lambda(\lambda-1)}{2(p-1)}\int_{Q}\Lambda(v, \nabla v)f^{\lambda-2}|\nabla f|^{2}\eta^{2}+\lambda c_{1}\int_{Q}f^{\lambda+1}\eta^{2}+\frac{\lambda c_{2}}{2}\int_{Q}yf^{\lambda}\eta^{2} \\&\leq \frac{\lambda C_{1}}{\lambda-1}\int_{Q}\Lambda(v, \nabla v)f^{\lambda}|\nabla \eta|^{2}+2\int_{Q}f^{\lambda}\eta\partial_{t}\eta.
\end{align*}
We have $$\left|\nabla \left(f^{\lambda/2}\eta\right)\right|^{2}\leq \frac{\lambda^{2}}{4}f^{\lambda-2}|\nabla f|^{2}\eta^{2}+f^{\lambda}|\nabla \eta|^{2}$$ and thus, for large enough $\lambda$,
\begin{align*}\left[ \int_{\Omega }f^{\lambda }\eta ^{2}\right] _{t_{1}}^{t_{2}}+&\frac{a_{p}\delta}{p-1}\int_{Q}\Lambda(v, \nabla v)\left|\nabla \left(f^{\lambda/2}\eta\right)\right|^{2}+\lambda c_{1}\int_{Q}f^{\lambda+1}\eta^{2}+\frac{\lambda c_{2}}{2}\int_{Q}yf^{\lambda}\eta^{2}\\&\leq \frac{a_{p}^{2}\delta+16(p-1)^{2}\delta}{(p-1)a_{p}}\int_{Q}\Lambda(v, \nabla v)f^{\lambda}|\nabla \eta|^{2}+2\int_{Q}f^{\lambda}\eta\partial_{t}\eta,
\end{align*} which implies (\ref{Cacciotype}).
\end{proof}

\subsection{Sobolev and Faber-Krahn inequalities}
\label{secMoser}

Let $M$ be a Riemannian manifold of dimension $n$.
Recall that on geodesically complete Riemannian manifolds, every geodesic ball $B$ is precompact. Consequently, the following \textit{Sobolev inequality} in $B$ of order $2$ holds: denoting by $r(B)$ the radius of $B$, we have for any non-negative function 
$w\in W_{0}^{1,2}(B)$,%
\begin{equation}
	\left( \int_{B}w^{2\kappa }\right) ^{1/\kappa }\leq S_{B}\int_{B}\left\vert
	\nabla w\right\vert ^{2},  \label{SBk}
\end{equation}%
where \begin{equation*}
	\kappa =\left\{ 
	\begin{array}{ll}
		\dfrac{n}{n-2}, & \text{if }n>2, \\ 
		\text{any number}>1, & \text{if }n\leq 2.%
	\end{array}%
	\right.  \label{k}
\end{equation*} and $S_{B}$ is called the \emph{Sobolev
	constant} in $B$. Clearly, the value of $\kappa $ is independent of $B$.
Let $\kappa ^{\prime }=\frac{\kappa }{\kappa -1}$
be the H\"{o}lder conjugate of $\kappa $ and set

\begin{equation}\label{nu}
	\nu =\frac{1}{\kappa^{\prime}}=\left\{ 
	\begin{array}{ll}
		\dfrac{2}{n}, & \text{if }n>2, \\ 
		\text{any number}<1, & \text{if }n\leq 2.%
	\end{array}%
	\right.
\end{equation}

It is known that if $M$ is complete and $Ric_{B}\geq -K$, for
some $K\geq 0$, then 
\begin{equation}
	S_{B}\leq Ce^{C_{n}\sqrt{K}r(B)}\frac{r(B)^{2}}{\mu(B)^{\nu}},  \label{lowerbdiota}
\end{equation}%
for positive constants $C,C_{n}$ (see \cite{Buser, grigor, Saloff}) and $\mu$ denotes the Riemannian measure on $M$.

\begin{lemma}[Moser inequality]\cite{grigor2023finite}
	\label{MoserLem}Let $w\in L^{2}\left( I;W_{0}^{1,2}(B)\right)$ be
	non-negative and $I$ be an interval in $\mathbb{R}_{+}$. Set $Q=B\times I$. 
	Then
	\begin{equation}
		\dint_{Q}w^{2\left( 1+\nu \right) }\leq S_{B}\left( \dint_{Q}\left\vert
		\nabla w\right\vert ^{2}\right) \sup\limits_{t \in I}\left( \dint_{B}w^{2}\right)
		^{\nu }.  \label{Moser}
	\end{equation}
\end{lemma}

\subsection{Comparison in two cylinders}

\begin{lemma}
	\label{Lemtwo}Consider two balls $B=B\left( x,r_{1}\right) $ and $B^{\prime}=B\left( x,r_{2 }\right) $ with $0<r_{2}<r_{1}$, and two cylinders%
	\begin{equation*}
		Q=B\times \lbrack t_{1},T],\ \ \ Q^{\prime }=B^{\prime }\times \left[ t_{2 },T\right],\end{equation*} where $0\leq t_{1}<t_{2}<T$.
	Assume that $\textnormal{Ric}_{B}\geq -K$ for some $K>0$.
	Let $u$ be a positive smooth solution of (\ref{olddtv}) in $B\times \lbrack t_{1},T]$ and let $f$ be defined by (\ref{DefF1}). Assume that \begin{equation*}\label{assumptionLamb}\Lambda_{min}\leq \Lambda(v, \nabla v)\leq \Lambda_{max}\quad \textnormal{in}~Q,\end{equation*} where $\Lambda(v, \nabla v)$ is defined by (\ref{defLamb}).
	Then, for any large enough $\lambda>1$,
	\begin{equation}
		\int_{Q^{\prime }}f^{\lambda \left( 1+\nu \right) }\leq \frac{CS_{B}}{\Lambda_{min}} \left[\frac{1}{t_{2}-t_{1}} +\frac{\Lambda_{max}}{(r_{1}-r_{2})^{2}}\right]^{ 1+\nu }\left(\int_{Q}f^{\lambda}\right)^{1+\nu},  \label{vQ'QQ}
	\end{equation}%
	where constant $C$ depends on $p$, $q$, $n$ and $\nu $ (where $\nu$ is given by (\ref{nu})), but it is
	independent of $\lambda $.
\end{lemma}

\begin{proof}
	Let us consider in $Q$ the function $\eta(x, t)=\eta_{1}(x)\eta_{2}(t)$, where $\eta_{1}$ is a bump function of $B^{\prime}$ in $B$ and \begin{equation*}
		\eta_{2}(t)=\left\{ 
		\begin{array}{ll}
			\frac{t-t_{1}}{t_{2}-t_{1}}, & t_{1}\leq t<t_{2}, \\ 
			1, & t_{2}\leq t\leq T.
		\end{array}%
		\right.
	\end{equation*} Set again $\alpha =\frac{\lambda }{2}$ and since $u$ is a smooth function, $f^{\alpha }\eta\in L^{2}\left([t_{1},T]; W_{0}^{1, 2}(B)\right)$. Hence, applying the Moser inequality (\ref{Moser}) with $w=f^{\alpha }\eta$
	and using $w^{2}=f^{\lambda }\eta ^{2},$
	we obtain
	\begin{equation*}
		\int_{Q}f^{\lambda \left( 1+\nu \right) }\eta ^{2\left( 1+\nu \right) }\leq
		S_{B} \int_{Q}\left\vert \nabla \left( f^{\alpha }\eta \right)
		\right\vert ^{2} \sup_{t\in \left[ t_{2},T\right] }\left(
		\int_{B}f^{\lambda }\eta ^{2}\right) ^{\nu }.
	\end{equation*}%
	By (\ref{Cacciotype}) we have%
	\begin{equation*}
		\Lambda_{min}\int_{Q}\left\vert \nabla \left( f^{\alpha }\eta \right) \right\vert
		^{2}\leq \frac{C_{1}}{c_{3}} \int_{Q}\left[\eta\partial_{t}\eta +\Lambda_{max}\left\vert \nabla \eta
		\right\vert ^{2}\right]f^{\lambda }
	\end{equation*}%
	and 
	\begin{equation*}
		\sup_{t\in \left[ t_{2},T\right] }\left( \int_{B}f^{\lambda }\eta ^{2}\right)
		\leq C_{1}\int_{Q}\left[\eta\partial_{t}\eta+ \Lambda_{max}\left\vert \nabla \eta
		\right\vert ^{2}\right]f^{\lambda },
	\end{equation*}%
	where in the latter we used that $\eta_{2}(t_{1})=0$.
	Therefore, it follows that 
	\begin{equation}\label{combmoser}
		\int_{Q}f^{\lambda \left( 1+\nu \right) }\eta ^{2\left( 1+\nu \right) }\leq
		\frac{C_{1}^{1+\nu}S_{B}}{c_{3}\Lambda_{min}}\left(\int_{Q}\left[\eta\partial_{t}\eta + \Lambda_{max}\left\vert \nabla \eta
		\right\vert ^{2}\right]f^{\lambda }\right)^{1+\nu}.
	\end{equation}
	Using that $\eta =1$ in $Q^{\prime }$, $\left\vert \nabla \eta
	\right\vert \leq \frac{1}{r_{1}-r_{2}}$ and $\partial_{t}\eta\leq \frac{1}{t_{2}-t_{1}}$ we obtain 
	\begin{align*}
		\int_{Q^{\prime }}f^{\lambda \left( 1+\nu \right) }\leq \frac{C_{1}^{1+\nu}S_{B}}{c_{3}\Lambda_{min}} \left[\frac{1}{t_{2}-t_{1}} +\frac{\Lambda_{max}}{(r_{1}-r_{2})^{2}}\right]^{ 1+\nu }\left(\int_{Q}f^{\lambda}\right)^{1+\nu}
	\end{align*}%
	which is (\ref{vQ'QQ}).
\end{proof}

\subsection{Iterations and the mean value theorem}
\label{SecMV}

\begin{lemma}
	\label{LemMoser}Let $B=B\left( x_{0},R\right) $ and $%
	T>0$. Suppose that $Ric_{B}\geq -K$, for some $K\geq 0$. Let $u$ be a positive smooth solution of (\ref{olddtv}) in $B\times \lbrack 0,T\rbrack$ and let $f$ be defined by (\ref{deff}).
	Let us set 
	\begin{equation*}
		Q=B\times \left[ 0,T\right] .
	\end{equation*}%
	Assume that in $Q$, \begin{equation}\label{condonLambda}\Lambda_{min}\leq \Lambda(v, \nabla v)\leq \Lambda_{max}.\end{equation}
	Then, for the cylinder 
	\begin{equation*}
		Q^{\prime }=\frac{1}{2}B\times [ \frac{1}{2}T,T] ,
	\end{equation*}%
	we have for all large enough $\lambda>1$,
	\begin{equation}
		\left\Vert f\right\Vert _{L^{\infty }(Q^{\prime })}\leq \left(\frac{CS_{B}}{\Lambda_{min}}\left[\frac{1}{T} +\frac{\Lambda_{max}}{R^{2}}\right]^{1+\nu}\right) ^{\frac{1}{\lambda\nu }}\left\Vert f\right\Vert _{L^{\lambda }(Q)},  \label{v31}\end{equation}%
	where $C=C\left( p,q, n, \nu \right) $.
\end{lemma}

\begin{proof}
	Consider, for $k\geq 0$, sequences 
	$$r_{k}=\left( \frac{1}{2}+2^{-(k+1)}\right) R\quad \textnormal{and} \quad t_{k}=\left(1-2^{-k}\right)T$$ and set 
	\begin{equation*}
		B_{k}=B\left( x_{0},r_{k}\right) ,\ \ \ \ Q_{k}=B_{k}\times \left[ t_{k},T\right]
	\end{equation*}%
	so that 
	\begin{equation*}
		\,B_{0}=B,\ \ \ \ Q_{0}=Q\ \ \ \ \text{and\ \ \ }Q_{\infty}:=\lim_{k\rightarrow \infty }Q_{k}=Q^{\prime }
	\end{equation*}%

	Set also $\lambda _{k}=\lambda \left( 1+\nu \right) ^{k}$
	and%
	\begin{equation*}
		J_{k}=\int_{Q_{k}}f^{\lambda _{k}}.
	\end{equation*}%
	By (\ref{vQ'QQ}) and using $r_{k}-r_{k+1}=2^{-\left( k+2\right) }R$ and $t_{k+1}-t_{k}=2^{-(k+2)}T$, we get
	\begin{align*}
		J_{k+1}& \leq \frac{CS_{B_{k}}}{\Lambda_{min}}\left[\frac{1}{t_{k+1}-t_{k}} +\frac{\Lambda_{max}}{(r_{k}-r_{k+1})^{2}}\right]^{1+\nu}J_{k}^{1+\nu } \\
		& = \frac{CS_{B_{k}}}{\Lambda_{min}}\left[\frac{2^{k+2}}{T}+\frac{2^{k+2}\Lambda_{max}}{R^{2}}\right]^{1+\nu}J_{k}^{1+\nu } \\
		& \leq A^{k}\Theta ^{-1}J_{k}^{1+\nu },
	\end{align*}%
	where 
	$A=\max\left(2^{1+\nu}, \left( 1+\nu \right) ^{1+\nu}\right)$
	and
	\begin{equation*}
		\Theta ^{-1}=\frac{CS_{B}}{\Lambda_{min}}\left[\frac{1}{T} +\frac{\Lambda_{max}}{R^{2}}\right]^{1+\nu}.
	\end{equation*}
	
	By Lemma \ref{LemJk} (see Appendix), we conclude that 
	\begin{align*}
		J_{k}& \leq \left( \left( A^{1/\nu }\Theta ^{-1}\right) ^{1/\nu
		}J_{0}\right) ^{\left( 1+\nu \right) ^{k}}\left( A^{-1/\nu }\Theta \right)
		^{1/\nu } \\
		& =A^{\frac{(1+\nu )^{k}-1}{\nu ^{2}}}\Theta ^{-\frac{(1+\nu )^{k}-1}{\nu }%
		}J_{0}^{(1+\nu )^{k}}.
	\end{align*}%
	It follows that%
	\begin{equation*}
		\left( \int_{Q_{k}}f^{\lambda _{k}}\right) ^{1/\lambda _{k}}\leq A^{\frac{%
				1-(1+\nu )^{-k}}{\lambda \nu ^{2}}}\Theta ^{-\frac{1-(1+\nu )^{-k}}{\lambda
				\nu }}\left( \int_{Q}f^{\lambda }\right) ^{1/\lambda }.
	\end{equation*}%
	As $k\rightarrow \infty $, we obtain%
	\begin{align*}
		\left\Vert f\right\Vert _{L^{\infty }(Q^{\prime })}& \leq A^{\frac{1}{\lambda\nu ^{2}}}\Theta ^{-\frac{1}{\lambda \nu }}\left\Vert f\right\Vert_{L^{\lambda }(Q)} \\
		& = A^{\frac{1}{\lambda\nu ^{2}}}\left(\frac{CS_{B}}{\Lambda_{min}}\left[\frac{1}{T}+ \frac{\Lambda_{max}}{R^{2}}\right]^{1+\nu}\right) ^{\frac{1}{\lambda\nu }}\left\Vert f\right\Vert _{L^{\lambda }(Q)} \\
		&=\left(\frac{CS_{B}}{\Lambda_{min}}\left[\frac{1}{T} +\frac{\Lambda_{max}}{R^{2}}\right]^{1+\nu}\right) ^{\frac{1}{\lambda\nu }}\left\Vert f\right\Vert _{L^{\lambda }(Q)},
	\end{align*}%
	where $A^{1/\nu}$ was absorbed into $C$. This implies (\ref{v31}) and completes the proof.
	\end{proof}

\begin{lemma}\label{normbound}
Under the assumptions of Lemma \ref{LemMoser} 
we have for any large enough $\lambda>1$, 
\begin{equation}
	\left\Vert f\right\Vert _{L^{\lambda(1+\nu) }(Q^{\prime })}\leq \left(\frac{CS_{B}}{\Lambda_{min}}\right)^{\frac{1}{\lambda(1+\nu)}}\frac{C_{1}}{c_{1}}\left[\frac{1}{T}+\frac{\Lambda_{max}\lambda}{R^{2}}\right]\left(T\mu(B)\right)^{1/\lambda},  \label{v31norm}\end{equation}%
where $C=C\left( p,q, n, \nu, \alpha \right) $.
\end{lemma}

\begin{proof}
Let $\eta=\eta_{1}(x)^{\lambda+1}\eta_{2}(t)^{\frac{\lambda+1}{2}}$, where $\eta_{1}$ is a bump function of $B^{\prime}$ in $B$ and $$\eta_{2}(t)=\left\{ 
\begin{array}{ll}
	\frac{2t}{T}, & 0\leq t<\frac{1}{2}T, \\ 
	1, & \frac{1}{2}T\leq t\leq T.
\end{array}%
\right.$$ Similarly to (\ref{combmoser}) we can show that \begin{equation}\label{subtracting}\int_{Q^{\prime}}f^{\lambda \left( 1+\nu \right)  }\leq
\frac{S_{B}}{\Lambda_{min}}\left(\int_{Q}\left[\frac{C_{1}}{c_{3}^{1/(1+\nu)}}\left(\eta\partial_{t}\eta +\Lambda_{max}\left\vert \nabla \eta \right\vert ^{2}\right)-\frac{c_{1}}{c_{3}^{1/(1+\nu)}}\lambda f\eta^{2}\right]f^{\lambda }\right)^{1+\nu}.\end{equation}
Note that $\partial_{t}\eta\leq \frac{2(\lambda+1)}{T}\eta^{\frac{\lambda-1}{\lambda+1}}$ and $|\nabla \eta|\leq \frac{4(\lambda+1)}{R^{2}}\eta^{\frac{\lambda}{\lambda+1}}$.
Let us use the Young inequality in the following form: for all $a, b>$ and any $\varepsilon>0$, $$ab\leq \varepsilon \frac{a^{s}}{s}+\varepsilon^{-\frac{s^{\prime}}{s}}\frac{b^{s^{\prime}}}{s^{\prime}},$$ where $s>1$ and $s^{\prime}=\frac{s}{s-1}$. Hence, we have by Hölder's and Young's inequality with $s=\frac{\lambda+1}{\lambda}$ and $s^{\prime}=\lambda+1$, \begin{align*}
\frac{C_{1}}{c_{3}^{1/(1+\nu)}}\int_{Q}\eta\partial_{t}\eta f^{\lambda}&\leq \frac{2C_{1}(\lambda+1)}{c_{3}^{1/(1+\nu)}T}\int_{Q}\eta^{\frac{2\lambda}{\lambda+1}}f^{\lambda}\\&\leq \frac{2C_{1}(\lambda+1)}{c_{3}^{1/(1+\nu)}T}\left(\int_{Q}\eta^{2}f^{\lambda+1}\right)^{\frac{\lambda}{\lambda+1}}|Q|^{\frac{1}{\lambda+1}}\\&\leq \frac{1}{2}\frac{c_{1}}{c_{3}^{1/(1+\nu)}}\int_{Q}\lambda\eta^{2}f^{\lambda+1}+\frac{2^{\lambda}C_{1}^{\lambda+1}}{c_{3}^{1/(1+\nu)}c_{1}^{\lambda}T^{\lambda+1}}|Q|
\end{align*} and similarly
\begin{align*}
\frac{C_{1}}{c_{3}^{1/(1+\nu)}}\Lambda_{max}\int_{Q}|\nabla \eta|^{2} f^{\lambda}&\leq \frac{4C_{1}\Lambda_{max}(\lambda+1)^{2}}{c_{3}^{1/(1+\nu)}R^{2}}\int_{Q}\eta^{\frac{2\lambda}{\lambda+1}}f^{\lambda}\\&\leq \frac{4C_{1}\Lambda_{max}(\lambda+1)^{2}}{c_{3}^{1/(1+\nu)}R^{2}}\left(\int_{Q}\eta^{2}f^{\lambda+1}\right)^{\frac{\lambda}{\lambda+1}}|Q|^{\frac{1}{\lambda+1}}\\&\leq \frac{1}{2}\frac{c_{1}}{c_{3}^{1/(1+\nu)}}\int_{Q}\lambda\eta^{2}f^{\lambda+1}+\frac{2^{\lambda}C_{1}^{\lambda+1}}{c_{3}^{1/(1+\nu)}c_{1}^{\lambda}}\left(\frac{\Lambda_{max}\lambda}{R^{2}}\right)^{\lambda+1}|Q|,
\end{align*} where $|\cdot|$ denotes the Riemannian measure in $M\times(0, \infty)$.

Therefore, combining with (\ref{subtracting}), we deduce \begin{align*}\left(\int_{Q^{\prime}}f^{\lambda \left( 1+\nu \right)  }\right)^{\frac{1}{1+\nu}}\leq \left(\frac{S_{B}}{\Lambda_{min}}\right)^{\frac{1}{1+\nu}}\frac{2^{\lambda}C_{1}^{\lambda+1}}{c_{3}^{1/(1+\nu)}c_{1}^{\lambda}}\left(\frac{1}{T^{\lambda+1}}+\left(\frac{\Lambda_{max}\lambda}{R^{2}}\right)^{\lambda+1}\right)|Q|\end{align*} which implies (\ref{v31norm}).
\end{proof}

Combining Lemma \ref{LemMoser} and Lemma \ref{normbound} we obtain the following result which implies the main theorem \ref{mainthmint} from the Introduction.

\begin{theorem}\label{mainthm}
Let $F_{\alpha}$ be defined by (\ref{DefF1}). Under the assumptions of Lemma \ref{LemMoser}, we have in the cylinder 
\begin{equation*}\label{defQp}
	Q^{\prime }=\frac{1}{4}B\times [ \frac{1}{4}T,T] ,
\end{equation*}%
for any $\alpha>1$, 
\begin{equation}\label{upperforf}
\left\Vert F_{\alpha}\right\Vert _{L^{\infty }(Q^{\prime })}\leq \left(\frac{CT}{\Lambda_{min}}\right)^{\frac{C^{\prime}}{1+\sqrt{K}R}} \frac{C_{1}}{c_{1}}\left[\frac{1}{T}+\frac{\Lambda_{max}\left(1+\sqrt{K}R\right)}{R^{2}}\right]^{1+\frac{C^{\prime}}{1+\sqrt{K}R }}+\frac{\alpha^{2}n K\delta^{2}\Lambda_{max}}{(p-1)(\alpha-1)},
\end{equation} where $C=C(p, q, n, \nu, \alpha)$ and $C^{\prime}=C^{\prime}(p, q, n, \nu, \alpha)$.
\end{theorem}		

\begin{proof}
Let $f=F_{\alpha}-\varphi$, where $\varphi$ is as in Lemma \ref{lemmavarphi}. Combining Lemma \ref{LemMoser} with $\lambda^{\prime}=\lambda(1+\nu)$ and Lemma \ref{normbound} in $Q_{1/2}=\frac{1}{2}B\times [ \frac{1}{2}T,T]$, we get 
\begin{align*}
\left\Vert f\right\Vert _{L^{\infty }(Q^{\prime })}&\leq \left(\frac{CS_{\frac{1}{2}B}}{\Lambda_{min}}\left[\frac{1}{T}+\frac{\Lambda_{max}}{R^{2}} \right]^{1+\nu}\right) ^{\frac{1}{\lambda\nu(1+\nu) }}\left\Vert f\right\Vert _{L^{\lambda(1+\nu) }(Q_{1/2})}\\&\leq \left(\frac{CS_{\frac{1}{2}B}}{\Lambda_{min}}\left[\frac{1}{T}+\frac{\Lambda_{max}}{R^{2}} \right]^{1+\nu}\right) ^{\frac{1}{\lambda\nu(1+\nu) }}\left(\frac{CS_{B}}{\Lambda_{min}}\right)^{\frac{1}{\lambda(1+\nu)}}\frac{C_{1}}{c_{1}}\left[\frac{1}{T}+\frac{\Lambda_{max}\lambda(1+\nu)}{R^{2}}\right]\left(T\mu(B)\right)^{1/\lambda}\\&\leq\left(\frac{CS_{B}}{\Lambda_{min}}\right)^{\frac{1}{\lambda\nu}}\frac{C_{1}}{c_{1}}\left[\frac{1}{T}+\frac{\Lambda_{max}\lambda}{R^{2}} \right]^{1+\frac{1}{\lambda\nu }}\left(T\mu(B)\right)^{1/\lambda}\\&=\left(\frac{CS_{B}\left(T\mu(B)\right)^{\nu}}{\Lambda_{min}}\right)^{\frac{1}{\lambda\nu}}\frac{C_{1}}{c_{1}}\left[\frac{1}{T}+\frac{\Lambda_{max}\lambda}{R^{2}}\right]^{1+\frac{1}{\lambda\nu }}.
\end{align*} 
Using the estimate (\ref{lowerbdiota}) for $S_{B}$, $$S_{B}\leq Ce^{C_{n}\sqrt{K}R}\frac{R^{2}}{\mu(B)^{\nu}},$$ we get $$\left\Vert f\right\Vert _{L^{\infty }(Q^{\prime })}\leq \left(\frac{Ce^{C_{n}\sqrt{K}R}T^{\nu}R^{2}}{\Lambda_{min}}\right)^{\frac{1}{\lambda\nu}}\frac{C_{1}}{c_{1}}\left[\frac{1}{T}+\frac{\Lambda_{max}\lambda}{R^{2}}\right]^{1+\frac{1}{\lambda\nu }}.$$ Taking now $\lambda=C^{\prime}\left(1+\sqrt{K}R\right)$ with $C^{\prime}=C^{\prime}(p, q, n, \nu, \alpha)$ large enough, we obtain $$\left\Vert f\right\Vert _{L^{\infty }(Q^{\prime })}\leq \left(\frac{CT}{\Lambda_{min}}\right)^{\frac{C^{\prime}}{1+\sqrt{K}R}} \frac{C_{1}}{c_{1}}\left[\frac{1}{T}+\frac{\Lambda_{max}\left(1+\sqrt{K}R\right)}{R^{2}}\right]^{1+\frac{C^{\prime}}{1+\sqrt{K}R }}.$$
By (\ref{uppervarphi}), we have $$\varphi(t)\leq \frac{\alpha^{2}n K\delta^{2}\Lambda_{max}}{(p-1)(\alpha-1)}.$$ Hence, we conclude $$\left\Vert F_{\alpha}\right\Vert _{L^{\infty }(Q^{\prime })}\leq \left(\frac{CT}{\Lambda_{min}}\right)^{\frac{C^{\prime}}{1+\sqrt{K}R}} \frac{C_{1}}{c_{1}}\left[\frac{1}{T}+\frac{\Lambda_{max}\left(1+\sqrt{K}R\right)}{R^{2}}\right]^{1+\frac{C^{\prime}}{1+\sqrt{K}R }}+\frac{\alpha^{2}n K\delta^{2}\Lambda_{max}}{(p-1)(\alpha-1)},$$ what was to be shown.
\end{proof}

\begin{corollary}
Assume that $Ric_{M}\geq -K$, for some $K\geq 0$, and suppose that (\ref{condonLambda}) holds in $M\times (0, \infty)$. Let $u$ be a positive smooth solution of (\ref{olddtv}) in $M\times(0, \infty)$. Then, for all $t>0$, \begin{equation}\label{globalupper}\left\Vert F_{\alpha}\right\Vert _{L^{\infty }(M)}\leq \frac{C_{1}}{c_{1}t}+\frac{\alpha^{2}n\delta^{2} K\Lambda_{max}}{(p-1)(\alpha-1)}.\end{equation}
\end{corollary}

\begin{proof}
Sending $R\to \infty$ in (\ref{upperforf}), we get (\ref{globalupper}).
\end{proof}

\section{Fast diffusion case}\label{fastdiff}

Let us set $$D=1-q(p-1).$$
Let $u$ be a positive smooth solution of (\ref{olddtv}). In the following we always assume that $D >0$. Define \begin{equation}\label{pressureD}v=\frac{q(p-1)}{D}u^{-\frac{D}{p-1}}.\end{equation} 
For any smooth function $\psi$ we define $$\mathcal{L}(\psi)=\func{div}\left( |\nabla v|^{p-2}A(\nabla \psi)\right),$$ where $$A(\nabla \psi)=\nabla \psi+(p-2)|\nabla v|^{-2}\langle \nabla v, \nabla \psi\rangle \nabla v.$$
Let us consider the operator $$\mathcal{F}=\partial_{t}-\frac{D}{p-1}v\mathcal{L}.$$

Let $\Omega$ be an open subset of $M$ and $I$ be an interval in $[0, \infty)$.

\begin{lemma}\label{bochnerlemD}\cite{wang2014gradient}
	Let $u$ be a positive smooth solution of (\ref{olddtv}) in $\Omega\times I$ and $v$ be defined by (\ref{pressureD}). Assume that $\textnormal{Ric}_{\Omega}\geq -K$, for some $K\geq0$. Then, for any $0<\alpha<1$, \begin{equation}\label{DefF1D}F_{\alpha}=\frac{|\nabla v|^{p}}{v}+\alpha\frac{\partial_{t}v}{v}=:y+\alpha z\end{equation} satisfies \begin{equation}\label{diffequationD}\mathcal{F}(F_{\alpha})\leq-\beta|\nabla v|^{p-2}\langle \nabla v, \nabla F_{\alpha}\rangle-c_{n}(y+z)^{2}-(p-1)(1-\alpha)z^{2}+c_{D}K|\nabla v|^{2(p-1)},\end{equation} where $\beta=p-2D$, $c_{n}=\frac{(p-1)(p-nD)}{nD}$ and $c_{D}=\frac{pD}{p-1}$.
\end{lemma}

From now on we also always assume that $p-nD>0.$

\begin{lemma}\label{lemmavarphiD}
	Under the assumptions of Lemma \ref{bochnerlemD}, set \begin{equation}\label{deffD}f(x, t)=F_{\alpha}(x, t)-\varphi(t),\end{equation} where $\varphi$ is a non-negative continuous function.
	Assume also that \begin{equation}\label{defLambD}\Lambda(v, \nabla v):=|\nabla v|^{p-2}v\leq \Lambda_{max}.\end{equation} Then we have \begin{align}\label{diffequationsecondD}\mathcal{F}(f)\leq-\beta|\nabla v|^{p-2}\langle \nabla v, \nabla f\rangle-c_{0}y^{2}+c_{2}yf+c_{2}y\varphi-c_{1}f^{2}-2c_{1}f\varphi-c_{1}\varphi^{2}+c_{D}K\Lambda_{max}y-\varphi(t)^{\prime},\end{align} where $c_{0}=\frac{(1-\alpha)(p-1)[p(1-\alpha)+\alpha nD]}{nD\alpha^{2}}$,  $c_{1}=\frac{(p-1)(p-\alpha nD)}{nD\alpha^{2}}$ and $c_{2}=\frac{2(p-1)p(1-\alpha)}{nD\alpha^{2}}$.
\end{lemma}

\begin{proof}
	From (\ref{diffequationD}) we have that $$\mathcal{F}(f)\leq-\beta|\nabla v|^{p-2}\langle \nabla v, \nabla f\rangle-c_{n}(y+z)^{2}-(p-1)(1-\alpha)z^{2}+c_{D}K|\nabla v|^{2(p-1)}-\varphi(t)^{\prime} .$$ Using that $z=\frac{1}{\alpha}(-y+f+\varphi)$, we get that \begin{align*}
		c_{n}(y+z)^{2}+(p-1)(1-\alpha)z^{2}=c_{0}y^{2}-c_{2}yf-c_{2}y\varphi+c_{1}f^{2}+2c_{1}f\varphi+c_{1}\varphi^{2}.
	\end{align*}
	Hence, $$\mathcal{F}(f)\leq-\beta|\nabla v|^{p-2}\langle \nabla v, \nabla f\rangle-c_{0}y^{2}+c_{2}yf+c_{2}y\varphi-c_{1}f^{2}-2c_{1}f\varphi-c_{1}\varphi^{2}+c_{D}K|\nabla v|^{2(p-1)}-\varphi(t)^{\prime}.$$ Since $|\nabla v|^{2(p-1)}\leq \Lambda_{max}y$ by (\ref{defLamb}), we conclude (\ref{diffequationsecondD}) what was to be shown.
\end{proof}

\begin{lemma}\label{lemcaccD}
	Let $f$ be defined by (\ref{deffD}) in a cylinder $\Omega\times I$, where $\varphi$ is the non-negative continuous function from Lemma \ref{lemmauxfuncD} (see Appendix) with $a$ and $b$ specified below. Assume that $\textnormal{Ric}_{\Omega}\geq -K$, for some $K\geq0$. Let $\eta \left( x,t\right) $ be a locally Lipschitz non-negative bounded
	function in $\Omega\times I$ such that $\eta \left( \cdot ,t\right) $ has
	compact support in $\Omega $ for all $t\in I$. Fix some large enough real $\lambda=\lambda(p, q, n, \alpha)>1$.
	Choose $t_{1},t_{2}\in I$ such that $t_{1}<t_{2}$ and set $Q=\Omega \times \left[ t_{1},t_{2}%
	\right] $. Then%
	\begin{align}\nonumber
		\left[ \int_{\Omega }f^{\lambda }\eta ^{2}\right] _{t_{1}}^{t_{2}}+&c_{3}
		\int_{Q}\left\vert \nabla \left( f^{\lambda/2 }\eta \right) \right\vert
		^{2}\Lambda(v, \nabla v)+\lambda \epsilon_{3}\int_{Q}f^{\lambda+1}\eta^{2}\\&\leq C_{1}\int_{Q}\left[\eta \partial _{t}\eta+\left\vert \nabla \eta \right\vert ^{2}\Lambda(v, \nabla v)\right]f^{\lambda }\label{CacciotypeD} ,
	\end{align}%
	where $\Lambda$ is defined by (\ref{defLamb}), $c_{3}=\frac{a_{p}D}{p-1}$, $\epsilon_{3}=\epsilon_{3}(p, q, n, \alpha)$,  $C_{1}=\frac{2a_{p}(p-1)+a_{p}^{2}D+32D(p-1)^{2}}{a_{p}(p-1)}$ and $a_{p}=\min(p-1, 1)$.
\end{lemma}

\begin{proof}
	Multiplying both sides of (\ref{diffequationsecondD}) with test function $\psi=f^{\lambda -1}\eta^{2}$ and integrating over $Q$ yields \begin{align}&\int_{Q}\left(\partial_{t}f-\frac{D}{p-1}v\mathcal{L}(f)\right)\psi\nonumber\\&\leq \int_{Q}\left(-\beta|\nabla v|^{p-2}\langle \nabla v, \nabla f\rangle-c_{0}y^{2}+c_{2}yf+c_{2}y\varphi-c_{1}f^{2}-2c_{1}f\varphi-c_{1}\varphi^{2}+c_{D}K\Lambda_{max}y-\varphi(t)^{\prime}\right)\psi\nonumber.\end{align} Using partial integration we obtain \begin{align*}-\int_{Q}v\mathcal{L}(f)\psi=\int_{Q}|\nabla v|^{p-2}\langle \nabla f, \nabla (v\psi)\rangle+(p-2)|\nabla v|^{p-4}\langle \nabla v, \nabla f\rangle \langle \nabla f, \nabla (v\psi)\rangle.\end{align*} Since $$\nabla \psi=(\lambda-1)f^{\lambda-2}\nabla f\eta^{2}+2f^{\lambda-1}\eta \nabla \eta,$$ we have \begin{align*}-\int_{Q}v\mathcal{L}(f)\psi=&\int_{Q}(\lambda-1)|\nabla v|^{p-2}f^{\lambda-2}v|\nabla f|^{2}\eta^{2}+2|\nabla v|^{p-2}vf^{\lambda-1}\eta\langle \nabla f, \nabla \eta\rangle\\&+ |\nabla v|^{p-2}f^{\lambda-1}\eta^{2}\langle \nabla f, \nabla v\rangle+(p-2)(\lambda-1)|\nabla v|^{p-4}\langle \nabla f, \nabla v\rangle^{2}f^{\lambda-2}\eta^{2}v\\&+(p-2)2|\nabla v|^{p-4}\langle \nabla f, \nabla v\rangle vf^{\lambda-1}\eta \langle \nabla v, \nabla \eta\rangle +(p-2) |\nabla v|^{p-2}\langle \nabla f, \nabla v\rangle f^{\lambda-1}\eta^{2}.\end{align*} Since \begin{align*}(\lambda-1)|\nabla v|^{p-2}f^{\lambda-2}v|\nabla f|^{2}\eta^{2}+&(p-2)(\lambda-1)|\nabla v|^{p-4}\langle \nabla f, \nabla v\rangle^{2}f^{\lambda-2}\eta^{2}v\\&\geq a_{p}(\lambda-1)|\nabla v|^{p-2}f^{\lambda-2}v|\nabla f|^{2}\eta^{2},\end{align*} where $a_{p}=1 $ if $p>2$ and $a_{p}=p-1 $ if $p<2$, we deduce \begin{align*} -\int_{Q}v\mathcal{L}(f)\psi\geq \int_{Q}a_{p}(\lambda-1)|\nabla v|^{p-2}f^{\lambda-2}v|\nabla f|^{2}\eta^{2}&-2(p-1)f^{\lambda-1}|\nabla v|^{p-2}v \eta |\nabla f||\nabla \eta|\\&+(p-1)|\nabla v|^{p-2}f^{\lambda-1}\eta^{2}\langle \nabla f, \nabla v\rangle.\end{align*}
	On the other hand, we have \begin{align*} &\int_{Q}\left(-\beta|\nabla v|^{p-2}\langle \nabla v, \nabla f\rangle-c_{0}y^{2}+c_{2}yf+c_{2}y\varphi-c_{1}f^{2}-2c_{1}f\varphi-c_{1}\varphi^{2}+c_{D}K\Lambda_{max}y-\varphi^{\prime}\right)\psi\\&\leq \int_{Q}\left(\beta|\nabla v|^{p-1}|\nabla f|-c_{0}y^{2}+c_{2}yf+c_{2}y\varphi-c_{1}f^{2}-2c_{1}f\varphi-c_{1}\varphi^{2}+c_{D}K\Lambda_{max}y-\varphi^{\prime}\right)f^{\lambda-1}\eta^{2}.\end{align*}
	Hence, combining these terms, we get
	\begin{align*}
		&\int_{Q}\partial_{t}ff^{\lambda-1}\eta^{2}+\frac{a_{p}D(\lambda-1)}{p-1}|\nabla v|^{p-2}f^{\lambda-2}v|\nabla f|^{2}\eta^{2}-2D f^{\lambda-1}|\nabla v|^{p-2}v \eta |\nabla f||\nabla \eta|\\&\leq\int_{Q}\left[ (\beta+D)|\nabla v|^{p-1}|\nabla f|-c_{0}y^{2}+c_{2}yf+c_{2}y\varphi-c_{1}f^{2}-2c_{1}f\varphi-c_{1}\varphi^{2}+c_{D}K\Lambda_{max}y-\varphi^{\prime}\right]f^{\lambda-1}\eta^{2}.
	\end{align*}
	Further, we have by Young's inequality, \begin{align*}2D& f^{\lambda-1}|\nabla v|^{p-2}v \eta |\nabla f||\nabla \eta|\\&\leq  \frac{a_{p}D(\lambda-1)}{2(p-1)}\Lambda(v, \nabla v)f^{\lambda-2}|\nabla f|^{2}\eta^{2}+\frac{8D(p-1)}{a_{p}(\lambda-1)}\Lambda(v, \nabla v)f^{\lambda}|\nabla \eta|^{2},\end{align*} for any $\epsilon_{1}>0$, \begin{align*}(p-D)&f^{\lambda-1}|\nabla v|^{p-1} \eta^{2} |\nabla f|\\&\leq \frac{(p-D)^{2}}{4\epsilon_{1}}\Lambda(v, \nabla v)f^{\lambda-2}|\nabla f|^{2}\eta^{2}+\epsilon_{1}yf^{\lambda}\eta^{2},\end{align*} and for any $\epsilon_{2}>0$, $$c_{D}K\Lambda_{max}y\leq \epsilon_{2}y^{2}+\frac{1}{\epsilon_{2}}(c_{D}K\Lambda_{max})^{2}.$$ 
	Choosing $\lambda$ large enough such that $$\frac{a_{p}D(\lambda-1)}{2(p-1)}-\frac{(p-D)^{2}}{4\epsilon_{1}}\geq \frac{a_{p}D(\lambda-1)}{4(p-1)},$$ we get 
	\begin{align*}\nonumber\int_{Q}&\partial_{t}ff^{\lambda-1}\eta^{2}+\frac{a_{p}D(\lambda-1)}{4(p-1)}\Lambda(v, \nabla v)f^{\lambda-2}|\nabla f|^{2}\eta^{2}\\\leq \int_{Q}& \frac{8D(p-1)}{a_{p}(\lambda-1)}\Lambda(v, \nabla v)f^{\lambda}|\nabla \eta|^{2}\\&+\left[-(c_{0}-\epsilon_{2})y^{2}+(c_{2}+\epsilon_{1})yf+c_{2}y\varphi-c_{1}f^{2}-2c_{1}f\varphi-c_{1}\varphi^{2}+\frac{(c_{D}K\Lambda_{max})^{2}}{\epsilon_{2}}-\varphi^{\prime}\right]f^{\lambda-1}\eta^{2}.\nonumber
	\end{align*}
	
	Let us now estimate the right hand side of this inequality.
	Again we have by Young's inequality, $$-(c_{0}-\epsilon_{2})y^{2}+((c_{2}+\epsilon_{1})f+c_{2}\varphi)y\leq \frac{((c_{2}+\epsilon_{1})f+c_{2}\varphi)^{2}}{4(c_{0}-\epsilon_{2})}$$ and thus, for any $\epsilon_{3}>0$, \begin{align*}&\nonumber-(c_{0}-\epsilon_{2})y^{2}+(c_{2}+\epsilon_{1})yf+c_{2}y\varphi-c_{1}f^{2}-2c_{1}f\varphi-c_{1}\varphi^{2}+\frac{(c_{D}K\Lambda_{max})^{2}}{\epsilon_{2}}-\varphi^{\prime}\\\leq&\nonumber-\epsilon_{3}f^{2}-\left(c_{1}-\epsilon_{3}-\frac{(c_{2}+\epsilon_{1})^{2}}{4(c_{0}-\epsilon_{2})}\right)f^{2}-\left(2c_{1}-\frac{c_{2}(c_{2}+\epsilon_{1})}{2(c_{0}-\epsilon_{2})}\right)f\varphi\\&\nonumber-\left(c_{1}-\frac{c_{2}^{2}}{4(c_{0}-\epsilon_{2})}\right)\varphi^{2} +\frac{(c_{D}K\Lambda_{max})^{2}}{\epsilon_{2}}-\varphi^{\prime}\\=&\nonumber-\epsilon_{3}f^{2}-\frac{4(c_{1}-\epsilon_{3})(c_{0}-\epsilon_{2})-(c_{2}+\epsilon_{1})^{2}}{4(c_{0}-\epsilon_{2})}f^{2}-\frac{4c_{1}(c_{0}-\epsilon_{2})-c_{2}(c_{2}+\epsilon_{1})}{2(c_{0}-\epsilon_{2})}f\varphi\\&-\frac{4c_{1}(c_{0}-\epsilon_{2})-c_{2}^{2}}{4(c_{0}-\epsilon_{2})}\varphi^{2}+\frac{(c_{D}K\Lambda_{max})^{2}}{\epsilon_{2}}-\varphi^{\prime}.
	\end{align*}
	Note that $$4c_{1}c_{0}-c_{2}^{2}=\frac{4(p-1)^{2}(1-\alpha)(p-nD)}{nD\alpha^{2}}>0.$$
	Hence, for small enough $\epsilon_{i}>0$, $i=1, 2, 3$, we see that all coefficients are positive. Therefore, \begin{align*}
		&-\frac{4(c_{1}-\epsilon_{3})(c_{0}-\epsilon_{2})-(c_{2}+\epsilon_{1})^{2}}{4(c_{0}-\epsilon_{2})}f^{2}-\frac{4c_{1}(c_{0}-\epsilon_{2})-c_{2}(c_{2}+\epsilon_{1})}{2(c_{0}-\epsilon_{2})}f\varphi\\&-\frac{4c_{1}(c_{0}-\epsilon_{2})-c_{2}^{2}}{4(c_{0}-\epsilon_{2})}\varphi^{2}+\frac{(c_{D}K\Lambda_{max})^{2}}{\epsilon_{2}}-\varphi^{\prime}\leq 0,
	\end{align*}
	if \begin{equation}\label{condforvarphi}-\frac{4c_{1}(c_{0}-\epsilon_{2})-c_{2}^{2}}{4(c_{0}-\epsilon_{2})}\varphi^{2}+\frac{(c_{D}K\Lambda_{max})^{2}}{\epsilon_{2}}-\varphi^{\prime}\leq 0.\end{equation}
	Let $\varphi$ be as in Lemma \ref{lemmauxfunc} with $a=\frac{4c_{1}(c_{0}-\epsilon_{2})-c_{2}^{2}}{4(c_{0}-\epsilon_{2})}$ and $b=\frac{(c_{D}K\Lambda_{max})^{2}}{\epsilon_{2}}$. Therefore, it satisfies (\ref{condforvarphi}) and we obtain
	\begin{align*}\int_{Q}&\partial_{t}ff^{\lambda-1}\eta^{2}+\frac{a_{p}D(\lambda-1)}{4(p-1)}\Lambda(v, \nabla v)f^{\lambda-2}|\nabla f|^{2}\eta^{2}\\\leq \int_{Q}& \frac{8D(p-1)}{a_{p}(\lambda-1)}\Lambda(v, \nabla v)f^{\lambda}|\nabla \eta|^{2}-\epsilon_{3}f^{\lambda+1}\eta^{2}.
	\end{align*}
	
	Also, we have $$\int_{Q}\partial_{t}ff^{\lambda-1}\eta^{2}=\frac{1}{\lambda}\int_{Q}\partial_{t} f^{\lambda}\eta^{2}=\frac{1}{\lambda}\left[ \int_{\Omega }f^{\lambda }\eta ^{2}\right] _{t_{1}}^{t_{2}}-\frac{2}{\lambda}\int_{Q} f^{\lambda}\eta\partial_{t}\eta.$$
	This yields \begin{align*}&\left[ \int_{\Omega }f^{\lambda }\eta ^{2}\right] _{t_{1}}^{t_{2}}+\int_{Q}\frac{a_{p}D\lambda(\lambda-1)}{4(p-1)}\Lambda(v, \nabla v)f^{\lambda-2}|\nabla f|^{2}\eta^{2}-2 f^{\lambda}\eta\partial_{t}\eta\\\leq &\int_{Q} \frac{8D\lambda(p-1)}{a_{p}(\lambda-1)}\Lambda(v, \nabla v)f^{\lambda}|\nabla \eta|^{2}-\epsilon_{3}\lambda f^{\lambda+1}\eta^{2}.
	\end{align*}
	We have $$\left|\nabla \left(f^{\lambda/2}\eta\right)\right|^{2}\leq \frac{\lambda^{2}}{4}f^{\lambda-2}|\nabla f|^{2}\eta^{2}+f^{\lambda}|\nabla \eta|^{2}$$ and thus, for large enough $\lambda$,
	\begin{align*}&\left[ \int_{\Omega }f^{\lambda }\eta ^{2}\right] _{t_{1}}^{t_{2}}+\int_{Q}\frac{a_{p}D}{p-1}\Lambda(v, \nabla v)\left|\nabla \left(f^{\lambda/2}\eta\right)\right|^{2}-2 f^{\lambda}\eta\partial_{t}\eta\\\leq &\int_{Q} \frac{a_{p}^{2}D+32D(p-1)^{2}}{a_{p}(p-1)}\Lambda(v, \nabla v)f^{\lambda}|\nabla \eta|^{2}-\epsilon_{3}\lambda f^{\lambda+1}\eta^{2},\end{align*} which implies (\ref{CacciotypeD}).
\end{proof}	

\begin{remark}
	By Remark \ref{remD} we have that \begin{align}\varphi(t)\leq \sqrt{\frac{b}{a}}&=\sqrt{\frac{4(c_{0}-\epsilon_{2})(c_{D}K\Lambda_{max})^{2}}{(4c_{1}(c_{0}-\epsilon_{2})-c_{2}^{2})\epsilon_{2}}}\label{uppervarphiD}.\end{align}	
\end{remark}

\begin{theorem}\label{mainthmD}
	Let $B=B\left( x_{0},R\right) $ and $%
	T>0$. Suppose that $Ric_{B}\geq -K$, for some $K\geq 0$. Let $u$ be a positive smooth solution of (\ref{olddtv}) in $B\times \lbrack 0,T\rbrack$ and let $F_{\alpha}$ be defined by (\ref{DefF1D}).
	Let us set 
	\begin{equation*}
		Q=B\times \left[ 0,T\right] .
	\end{equation*}%
	Assume that in $Q$ \begin{equation}\label{condonLambdaD}\Lambda_{min}\leq \Lambda(v, \nabla v)\leq \Lambda_{max}.\end{equation}
	Then, for the cylinder
	\begin{equation*}\label{defQp}
		Q^{\prime }=\frac{1}{4}B\times [ \frac{1}{4}T,T] ,
	\end{equation*}%
	we have for all $t>0$, any small enough $\epsilon_{2}, \epsilon_{3}>0$ as in the proof of Lemma \ref{lemcaccD} and any $0<\alpha<1$, 
	\begin{align}\nonumber
		\left\Vert F_{\alpha}\right\Vert _{L^{\infty }(Q^{\prime })}\leq& \left(\frac{CT}{\Lambda_{min}}\right)^{\frac{C^{\prime}}{1+\sqrt{K}R}} \frac{C_{1}}{\epsilon_{3}}\left[\frac{1}{T}+\frac{\Lambda_{max}\left(1+\sqrt{K}R\right)}{R^{2}}\right]^{1+\frac{C^{\prime}}{1+\sqrt{K}R }}\\&+\sqrt{\frac{4(c_{0}-\epsilon_{2})(c_{D}K\Lambda_{max})^{2}}{(4c_{1}(c_{0}-\epsilon_{2})-c_{2}^{2})\epsilon_{2}}},\label{upperforfD}
	\end{align} where $C=C(p, q, n, \nu, \alpha)$ and $C^{\prime}=C^{\prime}(p, q, n, \nu, \alpha)$.
\end{theorem}		

\begin{proof}
	Let $f=F_{\alpha}-\varphi$, where $\varphi$ is as in Lemma \ref{lemmavarphiD}. As in the proof of Lemma \ref{mainthm} we obtain 
	$$\left\Vert f\right\Vert _{L^{\infty }(Q^{\prime })}\leq \left(\frac{CT}{\Lambda_{min}}\right)^{\frac{C^{\prime}}{1+\sqrt{K}R}} \frac{C_{1}}{\epsilon_{3}}\left[\frac{1}{T}+\frac{\Lambda_{max}\left(1+\sqrt{K}R\right)}{R^{2}}\right]^{1+\frac{C^{\prime}}{1+\sqrt{K}R }},$$
	where $C=C(p, q, n, \nu, \alpha)$ and $C^{\prime}=C^{\prime}(p, q, n, \nu, \alpha)$.
	Using the upper bound (\ref{uppervarphiD}) for $\varphi$, we conclude (\ref{upperforfD}).
\end{proof}

\begin{corollary}
	Assume that $Ric_{M}\geq -K$, for some $K\geq 0$, and suppose that (\ref{condonLambdaD}) holds in $M\times (0, \infty)$. Let $u$ be a positive smooth solution of (\ref{olddtv}) in $M\times(0, \infty)$. Then, for all $t>0$ and any small enough $\epsilon_{2}, \epsilon_{3}>0$ as in the proof of Lemma \ref{lemcaccD} and any $0<\alpha<1$, \begin{equation}\label{globalupperD}\left\Vert F_{\alpha}\right\Vert _{L^{\infty }(M)}\leq \frac{C_{1}}{\epsilon_{3}t}+\sqrt{\frac{4(c_{0}-\epsilon_{2})(c_{D}K\Lambda_{max})^{2}}{(4c_{1}(c_{0}-\epsilon_{2})-c_{2}^{2})\epsilon_{2}}}.\end{equation}
\end{corollary}

\begin{proof}
	Sending $R\to \infty$ in (\ref{upperforfD}), we get (\ref{globalupperD}).
\end{proof}

\section{Appendix}

\subsection{Two auxiliary functions}\label{secauxfunc}

The following function is used in Lemma \ref{lemmavarphi}.

\begin{lemma}\label{lemmauxfunc}
For positive constants $A, a$ and $b$ consider the function	
\begin{equation}\label{defphi}\varphi(t)=\left\{ 
	\begin{array}{ll}
		\frac{1}{At+\frac{a}{2b}}, & 0<t<\frac{a(1-\varepsilon)}{2Ab}, \\ 
		\frac{b}{a+\sqrt{A}\tanh\left(\sqrt{A}bt+\sqrt{A}\xi\right)}, & t\geq \frac{a(1-\varepsilon)}{2Ab},
	\end{array}%
	\right.
\end{equation}
where $\xi\in \mathbb{R}$ and $0<\varepsilon<1$ are chosen such that $\varphi$ is continuous. Then $\varphi$ satisfies $$(1):\quad A\varphi^{2}+\varphi^{\prime}\geq 0\quad \textnormal{and}\quad\varphi\geq \frac{b}{a} \quad \textnormal{for}~0<t<\frac{a(1-\varepsilon)}{2Ab}$$ and $$(2):\quad A\varphi^{2}+\varphi^{\prime}\geq (a\varphi-b)^{2}\quad \textnormal{for}~t\geq \frac{a(1-\varepsilon)}{2Ab}.$$
\end{lemma}

\begin{proof}
Note that, for $0<t<\frac{a(1-\varepsilon)}{2Ab}$, $\varphi$ given by (\ref{defphi}) solves the differential equation $A\varphi^{2}+\varphi^{\prime}= 0$. Then we get for such $t$, $$\varphi(t)\geq \frac{2b}{(2-\varepsilon)a}\geq \frac{b}{a}.$$ 
Also, for $t\geq \frac{a(1-\varepsilon)}{2Ab}$, we see that $\varphi$ given by (\ref{defphi}) satisfies the differential equation $A\varphi^{2}+\varphi^{\prime}= (a\varphi-b)^{2}$. Hence, it remains to show that $\xi\in \mathbb{R}$ and $0<\varepsilon<1$ can be chosen such that $\varphi$ is continuous, that is,  \begin{equation}\label{forconti}\frac{2b}{(2-\varepsilon)a}=\frac{b}{a+\sqrt{A}\tanh\left(\frac{(1-\varepsilon)a}{2\sqrt{A}}+\sqrt{A}\xi\right)}.\end{equation} Indeed, since $\frac{2b}{(2-\varepsilon)a}>\frac{b}{a}$ and $\tanh(x)\to -1$ for $x\to -\infty$, we can choose $\xi\in \mathbb{R}$ and $0<\varepsilon<1$ so that (\ref{forconti}) is satisfied.
\end{proof}

\begin{remark}\label{remdel}\normalfont
Clearly, we have $$\varphi(t)\leq 2\frac{b}{a}.$$
\end{remark}	

The following function is used in Lemma \ref{lemcaccD}.

\begin{lemma}\label{lemmauxfuncD}
	For positive constants $a$ and $b$ consider in $\mathbb{R}_{+}$ the function	
	\begin{equation}\label{defphiD}\varphi(t)=\sqrt{\frac{b}{a}}\tanh\left(\sqrt{ab}(C+t)\right),
	\end{equation}
	where $C$ is a positive constant. Then $\varphi$ satisfies $$\varphi^{\prime}+a\varphi^{2}=b.$$
\end{lemma}

\begin{remark}\label{remD}\normalfont
	Clearly, we have $$\varphi(t)\leq \sqrt{\frac{b}{a}}.$$
\end{remark}	
		
\subsection{An auxiliary lemma}

\begin{lemma}[\cite{grigor2023finite}]
	\label{LemJk}Let a sequence $\left\{ J_{k}\right\} _{k=0}^{\infty }$ of
	non-negative reals satisfy%
	\begin{equation}
		J_{k+1}\leq \frac{A^{k}}{\Theta }J_{k}^{1+\omega }\ \ \text{for all }k\geq 0.
		\label{Jk+1}
	\end{equation}%
	where $A,\Theta ,\omega >0.$ Then, for all $k\geq 0$,%
	\begin{equation}
		J_{k}\leq \left( \frac{J_{0}}{\left( A^{-1/\omega }\Theta \right) ^{1/\omega
		}}\right) ^{\left( 1+\omega \right) ^{k}}\left( A^{-k-1/\omega }\Theta
		\right) ^{1/\omega }.  \label{Jk<}
	\end{equation}
\end{lemma}		
		
	\bibliographystyle{abbrv}
	\bibliography{gradleib}
	
	\emph{Universit\"{a}t Bielefeld, Fakult\"{a}t f\"{u}r Mathematik, Postfach
		100131, D-33501, Bielefeld, Germany}
	
	\texttt{philipp.suerig@uni-bielefeld.de}
\end{document}